\documentclass[a4paper,11pt]{article}
\usepackage{amsmath,amsthm,amssymb,url,verbatim,color}
\newcommand{\remove}[1]{}
\newtheorem{thm}{Theorem}[section]
\newtheorem{lem}[thm]{Lemma}

\newtheorem{cor}[thm]{Corollary}
\newtheorem{obs}[thm]{Observation}
\newtheorem{que}[thm]{Question}

\def\Fan{{\mathrm{Fan}}}
\def\odd{{\mathrm{odd}}}
\def\deg{{\mathrm{deg}}}
\def\c{c_\infty}
\def\sm{\setminus}

\newcommand{\e}{\epsilon}

\newcommand{\eps}{\varepsilon}
\newcommand{\mc}[1]{\mathcal{#1}}
\newcommand{\bb}[1]{\mathbb{#1}}
\newcommand{\brm}[1]{\operatorname{#1}}

\newcommand{\abs}[1]{| #1 |}

\title{Asymptotic Density of Graphs Excluding Disconnected Minors}
\author{Rohan Kapadia \thanks{Department of Computer Science and Software Engineering, Concordia University, Montreal, Quebec, Canada. Email: {\tt rohan.f.kapadia@gmail.com}} \and Sergey Norin\thanks{Department of Mathematics and Statistics, McGill University. Email: {\tt snorin@math.mcgill.ca}. Supported by an NSERC grant.} \and 
Yingjie Qian \thanks{School of Mathematics, Georgia Institute of Technology. Email: {\tt yingjie.qian@gatech.edu}}
}
	
\begin{document}
\maketitle
\baselineskip 20pt

\begin{abstract} 
For a graph $H$, let $$c_{\infty}(H)= \lim_{n \to \infty}\max\frac{|E(G)|}{n},$$
where the maximum is taken over all graphs $G$ on $n$ vertices not containing $H$ as a minor. Thus $c_{\infty}(H)$ is the asymptotic maximum density of graphs not containing $H$ as a minor. Employing a structural lemma due to Eppstein, we prove new upper bounds on $c_{\infty}(H)$ for disconnected graphs $H$. In particular, we determine  $c_{\infty}(H)$ whenever $H$ is union of cycles. Finally, we investigate the behaviour of $c_\infty(sK_r)$ for fixed $r$, where $sK_r$ denotes the union of $s$ disjoint copies of the complete graph on $r$ vertices.
Improving on a result of Thomason, we show that
$$c_\infty(sK_r)=s(r-1)-1 \mathrm{\; for \;} s ={\Omega}\left(\frac{\log{r}}{\log\log{r}}\right),$$ and
$$c_\infty(sK_r)>s(r-1)-1 \mathrm{\; for \;} s ={o}\left(\frac{\log{r}}{\log\log{r}}\right).$$
\end{abstract}

\section{Introduction}

A graph $H$ is \emph{a minor} of a graph $G$  if a graph isomorphic to $H$ can be obtained from a subgraph of $G$ by contracting edges. A well-studied  extremal question in graph minor theory is determining the maximum density of graphs $G$ not containing $H$ as a minor. We denote by $v(G)$ and $e(G)$ the number of  vertices and edges of a graph $G$, respectively, and by $d(G)=e(G)/v(G)$ the \emph{density} of a non-null graph $G$. 
Following Myers and Thomason~\cite{MyeTho05} for a  graph $H$ with $v(H) \geq 2$ we define  
\emph{the extremal function  $c(H)$ of $H$} as the supremum of $d(G)$ taken over all non-null graphs $G$ not containing $H$ as a minor. The asymptotic behaviour of $c(K_r)$, where $K_r$ denotes the complete graph on $r$ vertices, was studied in~\cite{Kostochka82,Kostochka84,Thomason84}, and  was determined precisely by Thomason~\cite{Thomason01}, who has shown that
\begin{equation}\label{e:Thomason}
c(K_r)=(\lambda+o_r(1))r\sqrt{\log{r}},
\end{equation}
where
$$\lambda = \max_{\alpha >0} \frac{1-e^{-\alpha}}{2\sqrt{\alpha}}=0.319...,$$
is an explicit constant, which we will refer to as  \emph{Thomason's constant}. 

In~\cite{Tho08} Thomason defined an asymptotic variant of the extremal function as  $$c_{\infty}(H)= \lim_{n \to \infty}\max_{v(G)=n} d(G)$$
where the maximum is taken over all graphs $G$ on $n$ vertices not containing $H$ as a minor. We refer to $c_{\infty}(H)$ as the \emph{asymptotic extremal function of $H$.} 
Clearly, $c_{\infty}(H) \leq c(H)$. When $H$ is connected then, as observed in~\cite{Tho08},   $c(H)=c_{\infty}(H)$, because in this case one can replace an $H$-minor free graph $G$ by a disjoint union of many copies of $G$ to obtain arbitrarily large $H$-minor free graphs with the same density as $G$. For disconnected graphs $H$ the parameters $c_{\infty}(H)$ and $c(H)$ frequently differ. 

Let $lH$ denote the union of $l$ disjoint copies of a graph $H$. The following theorem is the main result of~\cite{Tho08}.
\begin{thm}[Thomason~\cite{Tho08}]\label{t:Thomason}
	 $\:$
	 \begin{description}
		\item[a)] $c_{\infty}(lK_r)= (1+o_r(1))c(K_r)$  for fixed $l$,
		\item[b)] $c_{\infty}(lK_r)= l(r-1)-1$  for $l \geq 20c(K_r)$
	\end{description}
\end{thm}	

Powerful structural tools of graph minor theory become available when one considers large  graphs in minor-closed graph classes, and, in particular, when one investigates $c_{\infty}(H)$ rather than  $c(H)$. The main goal of this paper is to use one such tool, a lemma proved by Eppstein~\cite{Eppstein10}, to derive several new bounds on the asymptotic density of graphs excluding disconnected minors. In particular, we improve bounds in Theorem~\ref{t:Thomason}.  

Let us first present a natural lower bound on $c_{\infty}(H)$. 
Let $\tau(H)$ denote \emph{the vertex cover number} of the graph $H$, that is the minimum size of the set $X \subseteq V(H)$ such that $H - X$ is edgeless. 
Let $\bar{K}_{s,t}$ denote the graph obtained from the disjoint union of a complete graph $K_s$ and an edgeless graph $E_t$ on $t$ vertices by making every vertex of $K_s$ adjacent to every vertex of $E_t$. 
Then $\tau(\bar{K}_{s,t})=s$ for $t \geq 1$, and $\lim_{t \to \infty} d(\bar{K}_{s,t}) = s$. As the vertex cover of any minor of a graph $G$ does not exceed  $\tau(G)$, it follows that $H$ is not a minor of the graph $\bar{K}_{s,t}$ for any $s < \tau(H)$ and any $t$. Thus 
\begin{equation}\label{e:tau}
 c_{\infty}(H) \geq \tau(H)-1
\end{equation} 
for every graph $H$. We say that a graph $H$ is \emph{well-behaved} if (\ref{e:tau}) holds with equality. Dirac~\cite{Dirac64}, Mader~\cite{Mader68}, J{\o}rgensen~\cite{Jorgensen94}, and Song and Thomas~\cite{SonTho06} proved that $c(K_r)=r-2$ for $r \leq 5$, $r \leq 7$, $r=8$ and $r=9$, respectively. Thus $K_r$ is well-behaved for $r \leq 9$, however (\ref{e:Thomason}) implies that $K_r$ is far from being well-behaved for large $r$. On the other hand,  Theorem~\ref{t:Thomason} b) implies that $lK_r$ is well-behaved for fixed $r$ and large $l$. The results of this paper imply that  many classes of disconnected graphs are well-behaved, or are close to being well-behaved.

Our first result provides a general upper bound on $c_{\infty}(H)$ for a disconnected graph $H$ in terms of the asymptotic extremal function and the vertex cover number of its components. 

\begin{thm}
	\label{thm:infty+}
	Let $H$ be the disjoint union of non-null graphs $H_1$ and $H_2$, then 
	\begin{equation}\label{e:tau2}
	c_{\infty}(H) \leq \max\{\c(H_2), c_{\infty}(H_1) + \tau(H_2)\}.
	\end{equation}
	In particular,
	\begin{equation}\label{e:sum}
		\c(H) \leq \c(H_1)+\c(H_2)+1.
	\end{equation}
\end{thm}
     
Note that $\c(H)+l-1 \leq \c(lH)$ for any positive integer $l$ and non-null graph $H$.      
Theorem~\ref{thm:infty+} together with this observation immediately imply the following corollary, which establishes in a strong form Theorem~\ref{t:Thomason} a), and provides upper and lower bounds for $\c(lK_r)$ in terms of $c(K_r)$ which differ at most by a multiplicative factor of two.

\begin{cor}\label{c:lKrLoose}
For all positive integers $l$ and $r$ we have
$$
\max\{c(K_r)+l-1, l(r-1)-1\} \leq \c(lK_r) \leq c(K_r) + (l-1)(r-1).
$$
\end{cor}

Theorem~\ref{thm:infty+} also implies that if  $H_1$ is a well-behaved graph, and a graph $H_2$ satisfies $\c(H_2)~\leq~\c(H_1)~+~\tau(H_2)$, then the disjoint union of $H_1$ and $H_2$ is well-behaved. Thus the disjoint union of cliques of size nine or less is well-behaved.

The inequality (\ref{e:tau2}) does not necessarily hold with $c_{\infty}$ replaced by $c$. However, it was conjectured by the third author that (\ref{e:sum}) still holds. A weaker form of this conjecture has been verified by Cs\'{o}ka et al.~\cite{CLNWY} who have shown the following.
  
\begin{thm}[Cs\'{o}ka et al.~\cite{CLNWY}]\label{t:clnwy} Let $H$ be a disjoint union of $2$-connected graphs $H_1$,$H_2$,\ldots,$H_k$. Then
	$$
	c(H) \leq c(H_1)+c(H_2)+\ldots+c(H_k)+k-1.
	$$
\end{thm}

The proof of Theorem~\ref{t:clnwy} relies on extremal graph theory techniques, in particular, on a lemma about partitioning graphs into parts with prescribed average degree, the proof of which requires extensive calculations. In contrast, the proof of Theorem~\ref{thm:infty+} is very short, modulo the aforementioned lemma by Eppstein. 

In~\cite{CLNWY} Theorem~\ref{t:clnwy} is used to prove the following upper bound on the extremal function of the union of cycles, verifying conjectures of Reed and Wood~\cite{ReeWoo15}, and Harvey and Wood~\cite{HarWoo15}.

\begin{thm}\label{thm:cycles} Let $H$ be a disjoint union of cycles. Then
	\begin{equation}\label{e:cycledensity}
	c(H) \leq \frac{v(H)+\brm{comp}(H)}{2}-1.
	\end{equation}
\end{thm}

In the case when $H$ is the union of odd cycles, the right side of (\ref{e:cycledensity}) is equal to $\tau(H)-1$ and thus   the union of odd cycles is well-behaved. In most of the remaining cases for a union of cycles $H$, the exact value of $c(H)$ remains undetermined, but our next result completely  determines $c_{\infty}(H)$.

\begin{thm}\label{thm:2r}
	Let $H$ be a $2$-regular graph with $\brm{odd}(H)$ odd components. Then
	$$c_{\infty}(H)= \frac{v(H)+\brm{odd}(H)}{2}-1,$$
	unless $H=C_{2l}$, in which case $c_{\infty}(H)=l -\frac 12,$ or $H=kC_4$, in which case $c_{\infty}(H)=2k -\frac{1}{2}.$
\end{thm}

Next we turn to investigating unions of large cliques. Theorem~\ref{t:Thomason} b) and Theorem~\ref{thm:infty+} imply that for every $r$ there exists $l_0=l_0(r) \leq 20c(K_r)$, such that $lK_r$ is not well-behaved for $l<l_0$ and $lK_r$ is well-behaved for $l \geq l_0$. It follows from (\ref{e:Thomason}) that $l_0(r) \geq (\lambda+o_r(1))\sqrt{\log{r}}$. Thomason mentions in~\cite{Tho07,Tho08} that it is likely that $l_0(r) = \Theta(\sqrt{\log{r}})$. This prediction is motivated  by the belief that for large enough $r$ and any $l$, the extremal examples  should either be ``close" to being $K_r$-minor free or of the form $\bar{K}_{l(r-1)-1,n}$ for some $n$.   

We show that Thomason's prediction is almost, but not quite correct, as the next theorem implies that $l_0(r)={\Theta}(\log {r}/\log\log{r})$.  The main reason for the discrepancy is that for a certain range of $l$ we exhibit extremal examples, which do not have the structure suggested in~\cite{Tho08}, but are obtained by gluing certain non-uniform random graphs. 

\begin{thm}\label{thm:main} There exist constants $c,C>0$ such that for every positive integer $r$ 
	\begin{flalign*}
	&a)\qquad \c(lK_r) > l(r-1)-1 \qquad \mathrm{for} \qquad l \leq  c\frac{\log r}{\log\log r} \\
	&b)\qquad \c(lK_r) = l(r-1)-1 \qquad \mathrm{for} \qquad l \geq C\frac{\log r}{\log\log r}.
	\end{flalign*}
\end{thm}

Additionally, the next two theorems provide upper and lower bounds on $\c(lK_r)$, which allow us to approximate the error term $\c(lK_r) - lr$ in the range where this term is substantial, i.e. $l = o(\log {r}/\log\log{r})$.

\begin{thm}\label{thm:lower} Let $\lambda$ be  Thomason's constant. For $l=\omega(\sqrt{\log n})$ and $l =  o(\log {r}/\log\log{r})$ we have
	$$\c(lK_r) \geq  lr + (1-o(1)) 
	\frac{\lambda^2r\log r}{4l}.$$
\end{thm}

\begin{thm}\label{thm:upper} There exists a constant $C_u > 0$ such that for all positive integers $l,r$
	$$\c(lK_r) \leq  lr +  
	C_u\frac{r\log r}{l}.$$
\end{thm}

As one of the ingredients in the proof of Theorem~\ref{thm:upper} we need an upper bound on the extremal function $c(K_{s,t})$ which is within a constant factor of optimal. This extremal function has been extensively investigated in the past.
It follows from the results of Kostochka~\cite{Kostochka84} and Thomason~\cite{Thomason84} that $c(K_{s,t}) = O(t\sqrt{\log t})$ for all $s\leq t$. Myers~\cite{Myers03} considered $c(K_{s,t})$ for $s \ll t$ and conjectured that $c(K_{s,t}) \leq c_s t$ for some constant independent on $t$. K\"uhn and Osthus~\cite{KuhOst05} and Kostochka and Prince~\cite{KosPri08} have independently proved this conjecture by showing that $c(K_{s,t}) =(1/2+o(1))t$ for $s \ll t$.  Unfortunately, none of the above bounds suffice for our purpose and we prove the following result, which is tighter in the regime $s=\omega(t/\log t)$ and $s  = o(t)$. 

\begin{thm}\label{t:bipartite}
	Let $t\geq s \geq 2$ be positive integers. Then 
	$$c(K_{s,t}) \leq 40 (\sqrt{st \log s}+s+t).$$ 
\end{thm}

Theorem~\ref{t:bipartite} additionally answers a question of Harvey and Wood~\cite{HarWooAverage15}. They have asked whether there exists a constant $\eps > 0$ such that for every graph $H$ on $n$ vertices we have \begin{equation}\label{e:HW}
	c(H) \geq \eps n \sqrt{d(H - S)}\end{equation}
for some  set $S \subseteq V(H)$ such that $|S| \leq \frac{n}{\eps \log n}$. 
By Theorem~\ref{t:bipartite} the answer is negative. Indeed, if $s = \omega (t/\log t)$ then $d(K_{s,t}-S) \geq s/2$ for every  set $S$ as above. Therefore, if  additionally $s=o(t)$ then the bound given by Theorem~\ref{t:bipartite} is smaller than the bound in (\ref{e:HW}) by a factor of roughly $\sqrt{t/s}$.

The rest of the paper is structured as follows
In Section~\ref{s:eppstein} we introduce the lemma of Eppstein~\cite{Eppstein10}, which will serve as our main tool and prove several additional preliminary lemmas. We prove Theorem~\ref{thm:infty+} in Section~\ref{s:infty+}, and Theorem~\ref{thm:2r} in Section~\ref{s:2r}. In Section~\ref{s:lower} we prove a general lower bound on $c_{\infty}(lK_r)$ attained by a random construction and derive Theorem~\ref{thm:main} a) and~\ref{thm:lower} from this bound. In Section~\ref{s:tools} we introduce several additional tools we need for proving the upper bounds on $c_{\infty}(lK_r)$. In particular, we prove Theorem~\ref{t:bipartite}. In Section~\ref{s:main} we prove  Theorem~\ref{thm:main} a) and~\ref{thm:lower}. Section~\ref{s:conclude} contains the concluding remarks. 

\section{Blades, fans and Eppstein's lemma}\label{s:eppstein}

In this section we define blades and fans and present a lemma of 
Eppstein~\cite{Eppstein10}, which will provide the framework for proving our results. 

We say that a pair $(G,S)$ is \emph{a blade} if $G$ is a graph and $S \subsetneq V(G)$.
Given a blade $\mc{B}=(G,S)$ and a positive integer $k$, let $\Fan(\mc{B},k)$ or $\Fan(G,S,k)$ denote the graph obtained by $k$ copies of $G$ by identifying the vertices in $S$. For example, $\bar{K}_{s,t}$ can be considered as $\Fan(K_{s+1},S,t)$, where $S$ is a subset of vertices of $K_{s+1}$ of size $s$. 
It is easy to see that $$\lim_{k \to \infty} d(\Fan(G,S,k)) = \frac{e(G)-e(G[S])}{v(G)-|S|},$$
and we define the \emph{density of a blade $\mc{B}=(G,S)$} as $$d(\mc{B}) = d(G,S) = \frac{e(G)-e(G[S])}{v(G)-|S|}. $$

We say that a blade $(G,S)$  is \emph{semiregular} if \begin{itemize} 
	\item $G[S]$ is complete,
	\item $G \setminus S$ is connected,
	\item each vertex of $S$ has a neighbor in $ V(G) - S$, 
\end{itemize}
We say that a semiregular blade $\mc{B}$ is \emph{regular} if $\deg(v) \geq d(G,S)$ for every $v \in V(G\setminus S)$.

Given a graph $H$ and a blade $\mc{B}$, we say that $H$ is \emph{a minor of $\mc{B}$} if $H$ is a minor of $\Fan(\mc{B},k)$ for some $k$, and we say that $\mc{B}$  is \emph{$H$-minor free}, otherwise.

We are now ready to state the key lemma, which is proven in~\cite{Eppstein10} for general minor-closed classes of graphs.  For convenience we state only a weaker version for classes of graphs with a single excluded minor.

\begin{lem}[Eppstein~\cite{Eppstein10}]
		\label{lem:formfan}
 Let $H$ be a graph.  Then for any $\e>0$ there exists a regular $H$-minor free blade $\mc{B}$ such that $d(\mc{B}) \geq \c(H) - \e$.
\end{lem}

(In~\cite{Eppstein10}, it is only shown that a semiregular blade as above exists.  However, it is easy to that if a blade $(G,S)$ satisfies the conclusion of the lemma is chosen so that $d(G,S)$ is maximum and subject to that $v(G)$ is minimum, then $(G,S)$ is regular.)

Essentially, Lemma~\ref{lem:formfan} allows us to restrict our attention to  fans when proving upper bounds on $c_{\infty}(H)$. The following  convenient corollary is immediately implied by Lemma~\ref{lem:formfan}.

\begin{cor}\label{c:fanden}
Let $H$ be a graph, and let $c\in \bb{R}$ be such that $d(\mc{B}) \leq c$ for every regular  $H$-minor free blade $\mc{B}$. Then $c_{\infty}(H) \leq c$. 

Conversely, if $\mc{B}$ is an $H$-minor-free blade then $d(\mc{B}) \leq c_{\infty}(H).$ 
\end{cor}

We finish this section by introducing additional notation and several easy, but useful, lemmas. Let $\mc{B}=(G,S)$ be a blade. For $S' \subseteq S$, we denote by $\mc{B}[S']$ the blade $(G \setminus (S-S'),S')$ obtained from $\mc{B}$ by deleting vertices in $S - S'$. 

\begin{lem}\label{l:bladedensity} Let $\mc{B}=(G,S)$ be a blade, and let $S' \subseteq S$. Then $d(\mc{B}[S']) \geq d(\mc{B}) - |S|+|S'|$
\end{lem}	
\begin{proof}
We have $$
(e(G)-e(G[S]))-(e(G - (S - S'))-e(G[S'])) \leq  (|S|-|S'|)(v(G)-|S|),
$$	
implying the desired inequality by definition of the blade density. 
\end{proof}

\begin{lem}\label{l:bladetau} Let $(G,S)$ be a semiregular blade, and let $H$ be a graph. If $|S| \geq \tau(H)$ then $H$ is a minor of $G$.
\end{lem}

\begin{proof}
	Note that  $\bar{K}_{|S|,k}$  is a minor of $(G,S)$ for every $k$. On the other hand $H$ is isomorphic to a subgraph of $\bar{K}_{\tau(H),v(H) - \tau(H)}$. The desired conclusion follows.
\end{proof}	

Showing that a graph $G$ contains a graph $H$ as a minor typically involves constructing a model of $H$ in $G$, defined as follows. We say that a map $\mu$ is a \emph{blueprint of $H$ in $G$} if $\mu$ maps vertices of $H$ to disjoint subsets of vertices of $G$, called \emph{bags of $\mu$}. We will use $\mu(H)$ to denote $\cup_{v \in V(H)}\mu(v)$.

We say that a blueprint is a \emph{premodel} if for every edge $\{u,v\} \in E(H)$ there exists an edge of $G$ with one end in $\mu(u)$  and another in $\mu(v)$. Finally, we say that a premodel is a \emph{model} if $G[\mu(v)]$ is connected for every $u \in V(H)$. The following useful observation is well known.

\begin{obs}\label{o:model} A graph $H$ is a minor of a graph $G$ if and only if there exists a model of $H$ in $G$.
\end{obs}	

Observation~\ref{o:model} is used, in particular, in the proofs of the next remaining lemmas of this section.

\begin{lem}\label{l:subblade} Let $\mc{B}=(G,S)$ be a blade, let $H_1, H_2, \ldots, H_t$ be vertex disjoint graphs, and let $H$ be their union. 
	Then the following  are equivalent
	\begin{enumerate}
		\item $H$ is a minor of $\mc{B}$, and
		\item there exist disjoint  $S_1,\ldots, S_t \subseteq S$ such that  $H_i$ is a minor of $\mc{B}[S_i]$ for every $1 \leq i \leq t$.
	\end{enumerate}
\end{lem}	

\begin{proof}
We start by showing that the first condition implies the second. By Observation~\ref{o:model}, there exists a model $\mu$ of $H$ in $\Fan(G,S,k)$ for some positive integer $k$. Equivalently there exists models $\mu_i$ of $H_i$  in $\Fan(G,S,k)$ for $1 \leq i \leq t$ such that $\mu_i(H_i) \cap \mu_j(H_j) = \emptyset$ for all $i \neq j$. Let $S_i = \mu_i(H_i) \cap S$ then the second condition clearly holds.

The proof of the other implication is similar.  
\end{proof}	

\begin{lem}\label{l:minorcomplete} Let $\mc{B}=(G,S)$ be a semiregular blade. If $K_r$ is a minor of $\mc{B}$
then $K_r$ is a minor of $G$.	 
\end{lem}

\begin{proof} 
	Let $\mu$ be a model of $K_r$ in $\Fan(G,S,k)$ for some positive integer $k$. If $|S| \geq r$ then $K_r$ is a subgraph of $G[S]$ and so the lemma holds. Otherwise, there exists a $v \in V(K_r)$ such that $\mu(v) \subseteq V(G') \setminus S$ for some copy $G'$ of $G$ in  $\Fan(G,S,k)$. Then $\mu(u) \cap V(G') \neq \emptyset$  for every $u \in V(K_r)$, and it is easy to see that the restriction of $\mu$ to $V(G')$ is a model of $K_r$ in $G'$. 
\end{proof}

\section{Proof of Theorem~\ref{thm:infty+}}\label{s:infty+}
Let $c= \max\{\c(H_2), c_{\infty}(H_1) + \tau(H_2)\}$. By Corollary~\ref{c:fanden} it suffices to show that $d(\mc{B}) \leq c$ for every  $H$-minor free regular blade $\mc{B}=(G,S)$.

We number the vertices in $S=\{v_1,v_2,\dots,v_s\}$, where $s=|S|$. Let $S_i=\{v_1,\dots,v_i\}$, $\overline{S_i}=S-S_i$. Choose $i$ minimum such that $H_1$ is a minor of $\mc{B}[S_i]$. Thus  $\mc{B}[\overline{S_i}]$ is $H_2$ minor-free by Lemma~\ref{l:subblade}, and therefore $d(\mc{B}[\overline{S_i}]) \leq \c(H_2)$ by Corollary~\ref{c:fanden}.  
In particular, if $i=0$ then $d(G,S) \leq c_{\infty}(H_2) \leq c$, as desired. Thus we assume $i>0$. By Lemma~\ref{l:bladetau} we have  \begin{equation}\label{e:h2}
s-i \leq \tau(H_2)-1.
\end{equation}
By minimality of $i$, $\mc{B}[S_{i-1}]$  is $H_1$-minor-free. Therefore $\c(H_1) \geq d(\mc{B}[S_{i-1}])$. By Lemma~\ref{l:bladedensity} and (\ref{e:h2}), we have  $$d(G,S) \leq d(\mc{B}[S_{i-1}]) + s-i+1 \leq \c(H_1) +  \tau(H_2),$$ 
as desired.

\section{Proof of Theorem~\ref{thm:2r}}\label{s:2r}
A classical result of Erd\H{o}s and Gallai below implies that 
\begin{equation}\label{e:singlecycle}
\c(C_l) \leq \frac{l-1}2
\end{equation}
 for every  $l \geq 3$.
\begin{thm}[Erd\H{o}s and Gallai~\cite{ErdGal59}]\label{thm:ErdGal}
	Let $l \geq 3$ be an integer and let $G$ be a graph with $n$ vertices and more than $(l-1)(n-1)/2$ edges. Then $G$ contains a cycle of length at least $l$. 
\end{thm}

We prove  Theorem~\ref{thm:2r} by induction on $v(H)$. By (\ref{e:singlecycle}) we may assume that
 $H$ as at least $2$ components. 
Let
\[
  d_0=\begin{cases}
               2m-\frac{1}{2}, \text{ if $H=mC_4$;}\\
               \frac{v(H)+\odd(H)}{2}-1, \text{ otherwise.}
            \end{cases}
\] By Corollary~\ref{c:fanden} it suffices to show that $d(G,S) \leq d_0$ for every $H$-minor-free regular blade $\mc{B}=(G,S)$.
Let $C$ be the longest cycle in $H$, let $l = v(C)$, and let $H$ be the disjoint union of $C$ and a graph $H_1$.  

If $H_1$ is a minor of $G \sm S$ then $(G,S)$ is $C$-minor-free, and so by (\ref{e:singlecycle}) we have $d(G,S) \leq \frac{(l-1)}{2} \leq d_0$, as desired. Thus $H_1$ is not a minor of $G \sm S$, and by the induction hypothesis  
$$d(G\sm S)\leq\frac{v(H_1)+\odd(H_1)}{2}-\frac{1}{2}.$$
Suppose that $|S| \leq (l-1)/2$, then 
\begin{align*} %\label{eq:1}
d(G,S)&\leq |S|+ d(G\sm S) \leq |S|+ \frac{v(H_1)+\odd(H_1)}{2}-\frac{1}{2} \\ &\leq  \frac{v(H_1)+v(C)+\odd(H_1)}{2}-1 \leq \frac{v(H)+\odd(H)}{2}-1 \leq d_0.
\end{align*}
Thus we assume that $|S| \geq l/2$.

Suppose next that there exists $v \in V(G) - S$ such that $v$ is the only vertex in $V(G)-S$ adjacent to a vertex in $S$.  Then $e(G) - e(G[S]) \leq e(G \setminus S) +|S|$. If $|S| \geq |V(G)-S|$ then
 \begin{align*} %\label{eq:1}
d(G,S)&\leq  d(G\sm S) + \frac{|S|}{v(G)-|S|} \\ &\leq \frac{v(G)-|S|-1}{2} +\frac{|S|}{v(G)-|S|}  \\ &\leq  |S| \leq \tau(H)-1 \leq d_0,
\end{align*}
where second to last inequality uses Lemma~\ref{l:bladetau}.
Otherwise,
 \begin{align*} %\label{eq:1}
 d(G,S)&\leq  d(G\sm S) + \frac{|S|}{v(G)-|S|} \leq d(G\sm S) + 1 \\& \leq  \frac{v(H_1)+\odd(H_1)}{2}+\frac{1}{2}  \leq \frac{v(H)+\odd(H)}{2}-1 \leq d_0.
 \end{align*}
Thus we assume that there exists distinct $u_1,u_2 \in S$, $v_1,v_2\in G\sm S$ such that $u_1v_1,u_2v_2\in E(G)$.

Let $S' \subseteq S$ be such that $u_1,u_2 \in S'$, and let $k=|S'|$.  We show that $C_{2k+2}$ is a minor
$\mc{B}[S']$. Let $S'=\{u_1,u_2,\ldots,u_k\}$. We say that a path $P$ in $G$ is an \emph{$S'$-jump} if both ends of $P$ are in $S'$, and $P$ is otherwise disjoint from $S$.
By taking a path joining $v_1$ and $v_2$ in $G \sm S$ we obtain an $S'$-jump $P_1$ with ends $u_1$ and $u_2$ and at least $3$ edges. If $k=2$, then taking the union of two copies of $P_1$ in $\Fan(\mc{B}[S'],2)$ we obtain a cycle of length at least six, as desired. Thus we assume $k \geq 3$. Let $v_3$ be a neighbor of $u_3$ in $V(G) - S$, and assume without loss of generality that $v_3 \neq v_2$. Let $P_2$ be an  $S'$-jump of length at least three with ends $u_2$ and $u_3$. For $i=3,\ldots,k$, let $P_i$ be an $S'$-jump of length at least two with ends $u_i$ and $u_{i+1}$, where  $u_{k+1}=u_1$ by convention. By taking the union of copies of paths $P_1,\ldots,P_k$, each chosen from a separate copy of $G$ we obtain a cycle of length at least $2k+2$ in $\Fan(\mc{B}[S'],k)$, as desired.

We finish the proof by considering two cases. Suppose first that $H=mC_4$, and let $S'$ with $|S'|=2$ be as in the previous paragraph. Then $H_1=(m-1)C_4$ is not a minor of $\mc{B}[S-S']$ 
and therefore by Corollary~\ref{c:fanden}, Lemma~\ref{l:bladedensity} and the induction hypothesis we have
\begin{equation*} 
d(\mc{B})  \leq \c(H_1)+|S'| 
 \leq 2(m-1)-\frac{1}{2} +2  = d_0.
\end{equation*}
Thus we assume that at least one cycle in $H$ has length not equal to four.
If $l\geq5$, then by the claim above there exists $S' \subseteq S$ such that $|S'| \leq \lceil l/2\rceil -1 = (v(C)+\odd(C))/2 -1$, and $C$ is a minor of $\mc{B}[S']$. Again it follows that 
\begin{align*} 
d(\mc{B})  &\leq \c(H_1)+|S'| \leq\frac{v(H_1)+\odd(H_1)}{2}-\frac{1}{2} +|S'| \\ & 
\leq\frac{v(H)+\odd(H)}{2}-\frac{3}{2} < d_0.
\end{align*}
It remains to consider the case $l \leq 4$, but $H$ contains at least one cycle of length not equal  to four. It follows that $\c(H_1)\leq\frac{v(H_1)+\odd(H_1)}{2}-1$ by the induction hypothesis, and choosing $S' \subseteq S$ with $|S'|=2$, we once again have
\begin{equation*} 
d(\mc{B})  \leq \c(H_1)+|S'| \leq\frac{v(H_1)+\odd(H_1)}{2} +1 \leq\frac{v(H)+\odd(H)}{2}-1 =d_0,
\end{equation*}
finishing the proof.

\section{A lower bound on $\c(lK_r)$}\label{s:lower}

Our constructions of dense blades with no $lK_r$ minor are random. Let ${\bf G}(a,b,p,q)$ be a random graph, with $V({\bf G}(a,b,p,q))=A \cup B$, where $A$ and $B$ are disjoint sets with $|A|=a$, $|B|=b$, the vertices of $B$ form a clique and the edges are chosen independently at random so that every edge with both ends in $A$ is present with probability $p$ and an edge joining a vertex in $A$ to a vertex in $B$ is present with probability $q$.

The next lemma is a technical variation of a computation which to the best of our knowledge was first used by Bollobas, Caitlin and Erd\H{o}s~\cite{BCE80} to compute the size of the largest minor in a random graph. 

\begin{lem}\label{l:construction} Let positive integers $a,b$ and $r$, and  reals $\alpha,\beta >0$ be such that 	$a+b \leq r^2$, $r \leq 2b$ and\begin{equation}
  \alpha(r-b)b(\log (r-b)-\log\log r -3) \geq (\alpha a+\beta b)^2. \label{e:ab2/r2}   
	\end{equation} 
	Then $$\Pr [K_{r} \mathrm{\;is \;a \:minor \:of\;} {\bf G}(a,b,1-e^{-\alpha},1-e^{-\beta})  ] \leq e^{-2r\log{r}}.$$
\end{lem}	

\begin{proof} We denote the random graph ${\bf G}(a,b,1-e^{-\alpha},1-e^{-\beta})$ by $G$ for brevity. There are at most $$(a+b)^{r} {\leq} r^{2r} = e^{2r\log{r}}$$ blueprints $\mu$ of $K_r$ in $G$. Thus it suffices to show that the probability that for a fixed blueprint $\mu$  is a premodel of $K_r$ is at most $e^{-4r\log{r}}.$	
	
Let $\mc{K}_a$ be the collection of all bags of $\mu$ which lie completely in $A$, and let $\mc{K}_b$ be the collection of the remaining bags.  Let $\mc{K}_a = \{X_1,X_2,\ldots,X_s\}$, and let $x_i = |X_i|$ for $1 \leq i \leq s$. Note that $s\geq r-b$.  Let
$\mc{K}_b = \{U_1,U_2,\ldots,U_{r-s}\}$, and let $Y_i = U_i \cap A$, $Z_i = U_i \cap B$, $y_i = |Y_i|, z_i = |Z_i|$
for $1 \leq i \leq r-s$. Note that the probability that $X_i$ and $X_j$ are adjacent in $G$ is $1-e^{-\alpha x_ix_j}$, and the  probability that $X_i$ is adjacent to $U_j$ is $1-e^{-\alpha x_iy_j-\beta x_iz_j}$.

Suppose first that  $s > b$. We upper bound the probability that $\mu$ is premodel of $K_r$ by the probability that the bags in $\mc{K}_a$ are pairwise adjacent,  which is
$$
\prod_{1 \leq i < j \leq s}(1-e^{-\alpha x_ix_j}) \leq \exp\left(-\sum_{1 \leq i < j \leq s}e^{-\alpha x_ix_j}\right)
 $$
Thus it suffices to show that 
$$\sum_{1 \leq i < j \leq s}e^{-\alpha x_ix_j} \geq 4r\log{r}.$$
As $b \geq r-b$, the condition (\ref{e:ab2/r2}) implies that \begin{equation}\label{e:a2/b2}
b^2(\log b-\log\log r - 3) \geq \alpha a^2.
\end{equation} 

By the AM-GM inequality
\begin{align*}
\sum_{1 \leq i < j \leq s}e^{-\alpha x_ix_j} &\geq \binom{s}{2}\exp\left(-\frac{\alpha}{\binom{s}{2}}\sum_{1 \leq i < j \leq s}x_ix_j \right) \geq \frac{b^2}{2}\exp\left(-\alpha \left(\frac{a}{b} \right)^2 \right) \\ & \stackrel{(\ref{e:a2/b2})}{\geq} \frac{b^2}{2}\exp\left(-\log b+\log\log r + 3\right)\geq \frac{ 20b \log r}{2}{\geq}  4r\log{r},
\end{align*}
as desired.

Thus we assume that $s \leq b$. Now we upper bound the probability that every set in $\mc{K}_a$ is adjacent to every set in $\mc{K}_b$. Repeating the beginning of the argument in the previous case we see that it suffices to show that 
$$\sum_{1 \leq i \leq s} \sum_{1 \leq j \leq r-s}e^{-x_i (\alpha y_j + \beta z_j)} \geq 4r\log{r},$$
Let $x =\sum_{1 \leq i \leq s} x_i$. Applying the AM-GM inequality as before we obtain
\begin{align*}
\sum_{1 \leq i \leq s} &\sum_{1 \leq j \leq r-s}e^{-x_i (\alpha y_j + \beta z_j)}  \\
&\geq s(r-s) \exp\left(-\frac{1}{s(r-s)}\sum_{1 \leq i \leq s} x_i \left(\sum_{1 \leq j \leq r-s} \alpha y_j +  \sum_{1 \leq j \leq r-s} \beta z_j \right)\right)  \\ 
&\geq b(r-b) \exp\left(-\frac{x(\alpha(a-x)+\beta b)}{b(r-b)} \right) \geq b(r-b) \exp\left(-\frac{(\alpha a+\beta b)^2}{\alpha b(r-b)} \right)\\&\stackrel{(\ref{e:ab2/r2})}{\geq} b(r-b)\exp\left(-\log (r-b)+\log\log r + 3\right) = e^3 b \log r\geq 4r\log r,
\end{align*}
as desired.
\end{proof}	

\begin{thm}\label{t:lKrLower}
Let $\lambda$ be the Thomason's constant.	
There exists $\xi>0$ so that for every $0\leq \eps \leq 1/2$, $r \gg 1/\eps$ and 	$  \sqrt{\log r}/\eps \leq l \leq \log r$, we have
\begin{equation}
\label{e:lKrLower2}
\c(lK_r) \geq  lr + 
(1-\eps)\frac{\lambda^2r\log r}{4l} -lr\exp\left(-\frac{\xi\eps\log{r}}{l}\right).
\end{equation}
\end{thm}

\begin{proof} Consider $a,b,\alpha$ and $\beta$ satisfying the conditions of Lemma~\ref{l:construction}.
Let $G={\bf G}(a,l(b+1)-1,1-e^{-\alpha},1-e^{-\beta})$ be a random graph, and let the set of vertices $A$ and $B$ be as in the definition of such random graph.  
Consider the blade $\mc{B}=(G,B)$. If $lK_r$ is a minor of $\mc{B}$ then by Lemma~\ref{l:subblade} there exists $B' \subseteq B$ with $|B'| \leq \lfloor |B|/l \rfloor = b$ such that $K_r$ is a minor of $\mc{B}[B']$. 
From Lemma~\ref{l:minorcomplete}  it follows that $K_r$ is a minor of $G[A \cup B']$. However, by Lemma~\ref{l:construction} the probability that $K_r$ is a minor of $G[A \cup B']$ for some $B' \subseteq B$ with $|B'|=b$ is at most $$|B|^{b}\exp(-2r\log r) \leq (lr)^r \exp(-2r\log r) \leq \exp(r(\log\log r -\log r)) \leq e^{-r}.$$
Thus the probability that $lK_r$ is a minor of $\mc{B}$ is at most $e^{-r}$.
Let $$D(a,b,\alpha,\beta) = \frac{a}{2}(1-e^{-\alpha}) +  lb(1-e^{-\beta}).$$
An easy computation shows that 
$$\bb{E}[d(\mc{B})] = \frac{(a-1)}{2}(1-e^{-\alpha}) + (l(b+1)-1)(1-e^{-\beta}) \geq D(a,b,\alpha,\beta) + 1. $$
As $$\Pr [d(\mc{B}) \geq \bb{E}[d(\mc{B})] -1 ] \geq \frac{1}{a+l(b+1)} \geq \frac{1}{r^2} \geq e^{-r},$$
it follows that there exists an $lK_r$-minor-free blade $\mc{B}$ with density at least $D(a,b,\alpha,\beta)$, i.e. $c_\infty(lK_r) \geq D(a,b,\alpha,\beta)$. 

It remains to choose $a,b,\alpha$ and $\beta$ satisfying the conditions of Lemma~\ref{l:construction} so that $$D(a,b,\alpha,\beta)  \geq  lr + 
(1-2\eps)\frac{\lambda^2r\log r}{4l} -lr\exp\left(-\frac{2\xi\eps\log{r}}{l}\right).$$
(Note that we replaced $\eps$ by $2\eps$ for later convenience.)
Let constant $0<\alpha<1$ be chosen to maximize  $\frac{1-e^{-\alpha}}{2\sqrt{\alpha}}$, i.e. $\lambda=\frac{1-e^{-\alpha}}{2\sqrt{\alpha}}$, and let \begin{align*} &\gamma=\frac{\lambda(1-\eps)}{2},  &\sigma = \frac{\gamma r\log{r}}{l}, \\&k = \frac{\gamma^2r\log{r}}{l^2} =\frac{\gamma\sigma}{l},  &b = \lceil r-k \rceil, \\ &\beta =\frac{\eps\sqrt{\alpha}\sigma}{2r},
&a=\left\lceil \frac{(1-\eps)\sigma}{\sqrt{\alpha}} \right\rceil.
\end{align*}
Note that by the choice of $l$ we have
\begin{equation}
\gamma^2 \frac{r}{\log r} \leq k \leq  
\frac{\gamma^2 \eps r \log r}{\log r} = \frac{\eps r}{2}
\end{equation}
Let us first verify that $a,b,\alpha$ and $\beta$  satisfy  (\ref{e:ab2/r2}). For $r \gg 1/\eps$, we have
\begin{align*} 
\alpha&(r-b)b(\log (r-b)-\log\log r -3) \\
&\geq (1-\eps/2)^2\alpha rk\log r 
\\&=\left(\sqrt{\alpha}(1-\eps/{2})\sigma \right)^2
\end{align*}
Thus  it suffices to show that
$\alpha a + \beta b \leq  \sqrt{\alpha}(1-\eps/{2})\sigma$, which is immediate from the definitions.

We now return to the computation of $D(a,b,\alpha,\beta)$ for $a,b,\alpha$ and $\beta$ as above. Let $\xi = \sqrt{\alpha}\lambda/16$. We have
\begin{align*} 
D(a,b,\alpha,\beta) &\geq \frac{(1-\eps)\sigma}{2\sqrt{\alpha}}(1-e^{-\alpha}) + l\left(r-\frac{\gamma\sigma}{l}\right)(1-e^{-\beta}) \\ &\geq
2\gamma\sigma +lr - \gamma\sigma -lr\exp\left(-\frac{\eps\sqrt{\alpha}\lambda(1-\eps)\log{r}}{4l}\right) \\ 
&=
lr + \left(\frac{\lambda(1-\eps)}{2}\right)^2\frac{r \log r}{l} -lr\exp\left(-\frac{\eps\sqrt{\alpha}\lambda(1-\eps)\log{r}}{4l}\right)
\\ &\geq  lr + 
(1-2\eps)\frac{\lambda^2r\log r}{4l} -lr\exp\left(-\frac{2\xi\eps\log{r}}{l}\right),
\end{align*}
which finishes the proof of the theorem.
\end{proof}

\begin{proof}[Proof of Theorem~\ref{thm:main} a)]
	Let $\xi$ be as in Theorem~\ref{t:lKrLower}, and let $2 \sqrt{\log r} \leq l 
	\leq  \frac{\xi\log r}{2\log\log r}$. Thus $l = c\log r/\log\log r$ for some $c \leq \xi/2$. It suffices to show that $\c(lK_r) \geq lr$.
	By Theorem~\ref{t:lKrLower} applied with $\eps =1/2$ we have \begin{align*}
	\c(lK_r) - lr &\geq \frac{\lambda^2r\log r}{8l} -lr\exp\left(-\frac{\xi\log{r}}{2l}\right) \\ &=
	\frac{\lambda^2}{8c}r\log\log r - \frac{cr\log r}{\log\log r}e^{-\frac{\xi\log\log{r}}{2c}} \\ &\geq 
	  \frac{\lambda^2}{8c}r\log\log r - \frac{cr}{\log\log r} \geq 0,
	\end{align*}
	as desired.
\end{proof}
\begin{proof}[Proof of Theorem~\ref{thm:lower}]
The inequality (\ref{e:lKrLower2}) gives the required bound, as long as we show that for every $0<\eps \leq 1$ there exists $\delta >0$ so that for $l \leq \delta\log r/\log\log r$ we have $$ lr\exp\left(-\frac{\xi\eps\log{r}}{l}\right) \leq \eps \frac{r\log r}{l}. $$	

Let $\delta = \operatorname{min} \{\xi\eps,\sqrt{\eps}\}$. Then
\begin{align*}\exp&\left(-\frac{\xi\eps\log {r}}{l}\right)  \leq \exp\left(-\frac{\xi\eps\log\log r}{\delta}\right)\leq \frac{1}{\log r} \leq\frac{\eps (\log\log r)^2 }{\delta^2\log r} \leq \eps\frac{\log r}{l^2}, \end{align*}
as desired.
\end{proof}

\section{Hefty graphs}\label{s:tools}

In this section we introduce the tools which will be subsequently used to upper bound $\c(lK_r)$. These tools are built around the concept of hefty graphs. We say that a graph $H$ is \emph{hefty} if $H=K_2$, or $\deg(v) \geq 0.65|V(H)|$ for every $v \in V(H)$. (Our choice of  constant $0.65$ is motivated by Lemma~\ref{lem:ReedWood} below.) 

Classes of graphs with similar properties are considered in many proofs of upper bounds on the extremal function and the following lemmas demonstrate some of the ways in which they are used.

The first lemma  allows one to replace any graph by a hefty graph at a cost of a constant fraction of density. It is a variant of a result first proved by Mader~\cite{Mader68}, and appears in a slightly stronger form than the one stated below in Reed and Wood~\cite{ReeWoo15}.

\begin{lem}\label{lem:ReedWood}
	Let $G$ be a graph such that $d(G)\neq 0$. Then there exists a hefty minor $H$ of $G$  such that $|V(H)| \geq d(G)/2$.
\end{lem}

Next lemma shows that if a hefty graph $G$ contains a  small model of a graph $H$ and a graph $H'$ is obtained from $H$ by adding a few edges then $G$ contains a model of $H'$. 
We say that a set $F$ of pairs of vertices of $G$ is a \emph{completion} of a blueprint $\mu$ of a graph $H$ in a graph $G$ if $\mu$ is a model of $H$ in a graph obtained from $G$ by adding $F$ to $E(G)$. The \emph{defect} of a blueprint $\mu$ is the minimum size of a completion of $\mu$.

\begin{lem}\label{lem:heftymodel2}
	Let $G$ be a hefty graph, and let $\mu$ be a blueprint of a graph $H$ in $G$ with defect 
	at most $0.3|V(G)|-|\mu(H)|$. Then $\mu$ extends to a model of $H$ in $G$.
\end{lem}

\begin{proof}
	We prove the lemma by induction on the defect $c$ of $\mu$. The base case $c=0$ is immediate. For the induction step, let $F$ be a completion of $\mu$ with $|F|=c\geq 1$ and consider arbitrary $f=\{u,v\} \in F$. Note that $u$ and $v$ have at least $0.3|V(G)|$ common neighbors in $G$, and so there exists $w \in V(G) - \mu (H)$ adjacent to both $u$ and $v$. Let $x \in V(H)$ be such that $u \in \mu(x)$. Adding $w$ to $\mu(x)$ we obtain a blueprint $\mu'$ of $H$ in $G$ such that $F \setminus \{f\}$ is a completion of $\mu'$. By the induction hypothesis $\mu'$ extends to a model of $H$ in $G$ as desired.
\end{proof}

As a first application of the above lemmas we prove Theorem~\ref{t:bipartite}. The technical part of the proof is contained in the following lemma.
 
\begin{lem} \label{l:heftyminor3}
	Let $G$ be a hefty graph on $a$ vertices. Let $s,t,k,l$ be positive integers such that 
	$sk+tl \leq 3a/20$ and $(k-2)l-2 \geq \log_2 s$. Then $K_{s,t}$ is a minor of $G$.
\end{lem}

\begin{proof} Let $d=0.65$. For every  $v \in V(G)$ and a set $X \subseteq V(G)\setminus\{v\}$ of size $l$ chosen uniformly at random the probability that $v$ has no neighbor in $X$ is at most $(1 - d)^l$.  Thus for a set $X$ as above the expected number of vertices in $V(G)-X$ with no neighbor in $X$ is at most $a(1-d)^l$. We say that a set $X$ is \emph{good} if at most $3a(1-d)^l$ vertices in $V(G)-X$ have no neighbor in $X$. By Markov's inequality the probability that $X$ is good is at least $2/3$.
	
Given a good set $X$ if a set $Y$ of size $k$ is selected from $V(G)-X$ uniformly at random then the probability that no vertex of $Y$ is adjacent to a vertex of $X$ is at most $(4(1-d)^{l})^k < (1/2)^{(l-2)k}$. 
	
We now select disjoint subsets  $X_1,X_2,\ldots,X_{2t},Y_1,Y_2,\ldots,Y_{s}$ of $V(G)$ such that $|X_i|=l$, $|Y_j|=k$ uniformly at random. We say that a pair $(i,j)$ is \emph{fulfilled} if there exist $\{u,v\} \in E(G)$ with $u \in X_i$, $v \in V_j$. We say that $X_i$ is \emph{perfect} if $(i,j)$ is fulfilled for every $j$, and we say that $X_i$ is \emph{flawed} otherwise.

By the calculations above the probability that $X_i$ is good, but flawed  is at most $s(1/2)^{(l-2)k} \leq 1/4$. Therefore the probability that $X_i$ is perfect is at least $1/2$. Thus there exists a choice of subsets as above such that at least $t$ of subsets $X_1,X_2,\ldots,X_{2t}$ are perfect. If, say, $X_1,\ldots,X_t$ are these subsets then $Y_1,Y_2,\ldots,Y_{s},X_1,X_2,\ldots,X_{t}$ form a premodel $\mu$ of $K_{s,t}$ which can be extended to a model by Lemma~\ref{lem:heftymodel2}, as $2|\mu(K_{s,t})| \leq 2(sk+tl) \leq 3a/10$. 
\end{proof}

\begin{proof}[Proof of Theorem~\ref{t:bipartite}] Let $d= 40 (\sqrt{st \log s}+s+t)$, and let $G$ be a graph with $d(G) \geq d$. By Lemma~\ref{lem:ReedWood} there exists a hefty minor $H$ of $G$ with $a=|V(H)|\geq d/2$. Let $p = \sqrt{st \log_2 s}$, $k =\lceil p/s\rceil+2$, and $l =\lceil p/t\rceil+2$. Then we have
	\begin{align*}
	(l-2)k- 2&\geq (k-2)(l-2) \geq \frac{p^2}{st}=\log_2 s, \qquad \mathrm{and}\\
	sk+tl &< s(p/s+3)+ t(p/t+3) \\&= 2\sqrt{st \log_2 s}+3s+3t  \leq 3d/40 \leq 3a/20.
	\end{align*}
Thus $s,t,k$ and $l$ satisfy the conditions of Lemma~\ref{l:heftyminor3}. It follows that $K_{s,t}$ is a minor of $H$ as desired.
\end{proof}	

Next we prove a counterpart of Lemma~\ref{l:construction}. We will show that if a graph has the structure similar to that of the random examples of $K_r$ minor-free graphs considered in that lemma, but is somewhat denser, then it has a $K_r$ minor.

To make the above statement precise we need a definition. 
 We say that a partition $(A,B)$ of the vertices of the graph $G$ is \emph{$(a,b,\delta)$-semicomplete} if $|A|=a$, $|B|=b$, $G[A]$ is hefty, $G[B]$ is complete and every $v \in B$ has at least $(1 -\delta)a$ neighbors in $A$. We say that $G$ is
\emph{$(a,b,\delta)$-semicomplete} if $V(G)$ admits an $(a,b,\delta)$-semicomplete partition.  
We will investigate the range of parameters which guarantee the presence of a $K_r$ minor in an $(a,b,\delta)$-semicomplete graph. First, we need an easy lemma.

\begin{lem}\label{l:easyrandom}
	Let $G$ be a graph, let $d=e(G)/\binom{n}{2}$ and let $X \subseteq V(G)$, $|X|=k$  be chosen uniformly at random. Then $$\Pr \left[e(G[X]) \geq \left(d-\frac 12\right)\binom{k}{2}\right] \geq \frac{1}{2}.$$ 
\end{lem} 

\begin{proof}
	Note that the expected value of $e([G[X]])$ is $d \binom{k}{2}$, and so the lemma follows immediately from Markov's inequality.
\end{proof}

We are now ready to prove the first of the main results on minors in semicomplete graphs.

\begin{lem}\label{l:semicomplete1} There exists $\eps >0$ satisfying the following. Let $a,k,r$ be positive integers and $\delta>0$ be real so that
	\begin{align}
	k \cdot\brm \max \left\{\sqrt{\log k},-\frac{\log r}{\log{\delta}} \right\} < \eps a , \label{e:condition1}
	\end{align}
then every $(a,r-k,\delta)$-semicomplete graph has a $K_r$ minor.	
\end{lem} 
	
\begin{proof}
	Let $(A,B)$ be an $(a,r-k,\delta)$-semicomplete partition of vertices of a graph $G$.
	Let $0.05 \leq c \leq 0.1$ be such that $s=ca/k$ is an integer. We say that $X \subseteq A$ with $|X|=s$ is \emph{bad} if some vertex of $B$ has no neighbors in $X$, and good otherwise. Then the probability that a set $X$ chosen uniformly at random is bad is at most
	$$r \delta^s \leq r\delta^{\frac{a}{20k}} \leq \frac{1}{3},$$
	where the last condition follows from  (\ref{e:condition1}), when $\eps$ is sufficiently small.
	
	We now choose disjoint subsets $X_1,X_2,\ldots,X_{3k},Z$ of $A$ such that $|X_i|=s$,
	$|Z|=ks$ uniformly at random. By the computation above with probability greater than $1/2$ at least $k$ of the sets $X_1,X_2,\ldots,X_{3k}$ are good. By Lemma~\ref{l:easyrandom} with probability at least $1/2$ we have $d(G[Z]) \geq 0.15\cdot(a/20 -1) \geq a/200.$  
	
It follows that for some choice as above, $X_1,\ldots,X_k$ are good and $d(G[Z]) \geq a/200.$ By (\ref{e:condition1}) and (\ref{e:Thomason}) if $\eps$ is sufficiently small then there exists a model $\mu$ of $K_k$ in $G[Z]$. Assume for convenience that $V(K_k)=\{1,2,\ldots,k\}$, and  extend $\mu$ to a blueprint $\mu'$ of $K_k$ in $G[A]$ by adding $X_i$ to $\mu(i)$. Then $|\mu'(K_k)| \leq 2ca \leq 0.2a$ and the defect of $\mu'$ is at most $ca$. By Lemma~\ref{lem:heftymodel2} the blueprint $\mu'$ extends to a model $\mu''$ of $K_k$ in $G[A]$, and by the choice of $X_1,\ldots,X_k$ every vertex in $B$ has a neighbor in $\mu(i)$ for every $i$. Therefore adding each vertex of $B$ as a new bag to $\mu''$ produces a model of $K_r$ in $G$,  as desired.
\end{proof}

The next lemma differs from Lemma~\ref{l:semicomplete1} by the restriction on parameters and the construction of the model of $K_r$.

\begin{lem}\label{l:semicomplete2} There exists $\eps >0$ satisfying the following. Let $a,r\geq k \geq 2$ be positive integers and $\delta>0$ be real so that
\begin{equation}\label{e:condition4}
	\max\{r,\sqrt{rk\log{r}}\} < \eps a,
\end{equation}
	then every $(a,r-k,0.8)$-semicomplete graph has a $K_r$ minor.	
\end{lem}

\begin{proof}
	Let $(A,B)$ be an $(a,r-k,0.8)$-semicomplete partition of vertices of a graph $G$. As in the proof Lemma~\ref{l:semicomplete1} we can find $Z \subseteq A$ such that $d(G[Z]) \geq a/200$ and $|Z| \leq a/10$. (In fact, the constants can be significantly improved, if needed.) By Theorem~\ref{t:bipartite} and (\ref{e:condition4}) if $\eps$ is sufficiently small then $G[Z]$ contains a model of $K_{k,r}$ and thus a model of $\bar{K}_{k,r}$. Let  the vertices of independent set of $\bar{K}_{k,r-k}$ be $v_1,v_2,\ldots,v_{r-k} $ and let $B=\{u_1,u_2,\ldots,u_{r-k}\}$. As every vertex in $B$ has at least $a/5$ neighbors in $A$ and $|B| \leq a/10$, there exist distinct $x_1,\ldots,x_{r-k}$ in $A \setminus \mu(K_{k,r-k})$ such that $x_i$ is adjacent to $u_i$. By Lemma~\ref{lem:heftymodel2} the model  $\mu$ extends to a model of $\mu'$ of  $\bar{K}_{k,r-k}$ in $G[A]$ such that $x_i \in \mu'(v_i)$ for $1 \leq i \leq r-k$. Adding $u_i$ to $\mu'(v_i)$ for each $i$  produces the desired model of $K_r$ in $G$. 
\end{proof}

\section{Proof of Theorems~\ref{thm:main}  b) and Theorem~\ref{thm:upper}}\label{s:main}

We start this section by introducing a crucial lemma which will allow us to apply the results of the previous section. Recall that by Lemma~\ref{lem:ReedWood} every graph can be replaced with a hefty minor while losing only constant fraction of density. Given a blade $(G,S)$, we would like to apply it to the graph $G - S$ while controlling the loss of the density of the blade. We can do this if we first ensure that every vertex of $G-S$ has a large number of neighbors in $S$. This is accomplished by the next lemma.

First let us recall some standard definitions, which are used in the proof.
A \emph{separation} of a graph $G$ is a pair $(A,B)$ such that $A \cup B = V(G)$ and no edge of $G$ has one end in $A-B$ and the other in $B-A$. The \emph{order} of a separation $(A,B)$ is $|A \cap B|$. For $X, Y \subseteq V(G)$ an \emph{$(X,Y)$-linkage} is a set of vertex disjoint paths, 
each with one end in $X$ and the other end in $Y$. By Menger's theorem the maximum order of an $(X,Y)$-linkage in $G$ is equal to the minimum order of a separation $(A,B)$ of $G$ such that $X \subseteq A$, $Y \subseteq B$.

\begin{lem}\label{lem:densemodel}
For every graph $G$ there exists a  graph $H$ and  a model $\mu$ of a graph $H$ in $G$ such that for every $v \in V(H)$ there exists $u \in \mu(v)$ such that $\deg_G(u) \leq 96d(H)+24$.
\end{lem}

\begin{proof} 
Let $G_1\subseteq G$ be chosen such that $d(G_1)$ is maximum.  Let $d=d(G_1)$. 
Let $R$ be the set of all vertices of $G$ of degree at most $12d$, and let $\mathcal{P}$ be the $(V(G_1),R)$-linkage in $G$ of maximum order. Let $x=\abs{\mathcal{P}}$ and $n=\abs{V(G_1)}$. As noted above, by Menger's theorem there exists a separation $(A,B)$ of $G$ such that $V(G_1)\subseteq A,R\subseteq B$ and $\abs{A\cap B}=x$. Let $G_2=G[A]$ and $n'=\abs{A}$. Then, $\abs{A-B}=n'-x$ and every vertex in $A-B$ has degree at least $12d$. By the choice of $G_1$ we have
$$d\geq d(G_2) =\frac{e(G_2)}{v(G_2)}\geq\frac{6d(n'-x)}{n'}.$$
Thus $x\geq \frac{5}{6}n' \geq \frac{5}{6}n$. 

Let $Q$ be the set of starting vertices of paths $\mathcal{P}$ in $V(G_1)$, then $\abs{Q}=x$. Let $G_3=G[V(G_1)-Q]$, then 
$$e(G_3)\leq d v(G_3)=d(v(G_1)-\abs{Q})\leq dn/6.$$ 
Let $G_4=G_1 \setminus E(G_3)$, then $\abs{E(G_4)}\geq \frac{5}{6}dn$. 

Let $S$ be the set of vertices in $V(G_4)-Q$ with degree at least  $2d$. We claim that there exists a matching $M$ in $G_4$ so each vertex of $S$ is joined by an edge of $M$ to a vertex in $Q$. Suppose not. Then by Hall's theorem there exists a set $S' \subseteq Q$ such that $|S'| \leq |S|$ and all the edges of $G_4$ incident with vertices of $S$ have their second end in $S'$. It follows that $|E(G[S \cup S'])| \geq 2d|S|> d|S \cup S'|$ contradicting the choice of $d$ and proving our claim. For every edge of $e \in M$ with an end $q \in Q$  extend  the path $P$ in $\mathcal{P}$ which ends in $q$ to include $e$. 

We are now ready to construct the graph $H$ satisfying the lemma. Let $G_5=G_4[Q \cup S]$, let $V(H)= \mathcal{P} $  and $P',P'' \in H$ are adjacent in $H$ if some edge of $G_5$ joins a vertex of $P'$ to a vertex of $P''$. Then the identity map $\mu$ is a model of $H$ in $G$. 

Next we estimate $d(H)$. Note that $|V(P) \cap V(G_5)| \leq 2$ for every $P \in \mc{P}$, and every vertex of $G_5$ is a vertex of some path in $\mc{P}$. It follows that $e(H) \geq \frac{e(G_5)-v(H)}{4}$. 
Moreover, $$e(G_5) \geq e(G_4) - 2d(v(G_4)-|Q|-|S|) \geq \frac{5}{6}dn - 2d\frac{n}{6} \geq \frac{dv(H)}{2}.$$ 
Thus $d(H) \geq \frac{d-2}{8}.$ Finally, by the choice of $\mc{P}$,  for every $v \in V(H)$ there exists $u \in V(\mu(v))$ such that $\deg_G(u) \leq 12d \leq 96d(H)+24$.
\end{proof}

 We say that a blade $(G,S)$ is \emph{$(a,m)$-hefty} if \begin{itemize}
	\item $(G,S)$ is semiregular,
	\item $G\setminus S$ is hefty, 
	\item $a=|V(G) - S|$,
	\item there are at least $m$ edges joining vertices of $S$ to vertices of  $G\setminus S$.  \end{itemize}  We say that a blade $(G',S')$ is a \emph{minor} of a blade $(G,S)$ if $G'$ is obtained from $G$ by repeatedly deleting vertices and deleting and contracting edges with both ends in $V(G)\setminus S$. Lemmas~\ref{lem:densemodel} and~\ref{lem:ReedWood} imply the following.

\begin{lem}\label{lem:heftyblade} There exists a constant $D$ satisfying the following Let $\mc{B}=(G,S)$ be a regular blade such that $|V(G)-S|>1$. Then $\mc{B}$ has an $(a,d(\mc{B})a - Da^2)$-hefty minor for some positive integer $a \geq 2$. 
\end{lem}

\begin{proof}  We show that $D=204$ satisfies the lemma.
By Lemma~\ref{lem:densemodel} there exists a  graph $H$ and  a model $\mu$ of $H$ in $G \setminus S$ such that for every $v \in V(H)$ there exists $u \in \mu(v)$ such that $\deg_{G \setminus S}(u) \leq 96d(H)+24$.	As $\mc{B}$ is regular, it follows that each such vertex $u$ has at least $d(\mc{B}) - 96d(H)-24$ neighbors in $S$. Contracting the bags of $\mu$ to single vertices we obtain a minor $(G',S)$ of $\mc{B}$ such that $G' \setminus S$ is isomorphic to $H$ and every vertex in $V(G') \setminus S $  has at least $d(\mc{B}) - 96d(H)-24$ neighbors in $S$. Applying Lemma~\ref{lem:ReedWood} to $G' \setminus S$ we obtain a minor $(G'',S)$ of $\mc{B}$ such that $G'' \setminus S $ is hefty, $a=|V(G'')-S| \geq d(H)/2$ and every vertex $V(G'') \setminus S $  has at least $d(\mc{B}) - 96d(H)-24 \geq d(\mc{B}) - 204a$ neighbors in $S$.  Let $Z$ be the set of vertices in $S$ with no neighbors in $V(G'')-S$, then $(G'' - Z, S-Z)$ is $(a,d(\mc{B})a - 204a^2)$ hefty, as desired.
\end{proof}

Lemma~\ref{lem:heftyblade} allows us  to restrict our attention to hefty blades during the investigation of $c_\infty(lK_r)$ at the expense of an error term linear to the size of $V(G) \setminus S$ in such a blade $(G,S)$. Meanwhile, Lemmas~\ref{l:semicomplete1} and~\ref{l:semicomplete2} seem tailored for finding disjoint complete minors in hefty blades.
Our next lemma makes the connection explicit.

\begin{lem}\label{lem:heftyblade1}
Let $a,r \geq k \geq 2$ be positive integers. If every $(a,r-k,\frac{k-1}{r-1})$-semicomplete graph has a $K_r$ minor.
Then for every $l \geq 1$,
every $(a,(l(r-k)+k-1)a)$-hefty blade has an $lK_{r}$ minor. 
\end{lem}

\begin{proof} 
	
	Let $\mc{B}=(G,S)$ be an $(a,(l(r-k)+k-1)a)$-hefty blade. Let $H=G\setminus S$.  Let $\delta=(k-1)/(r-1)$, and let $S'$ be the set of all vertices in $S$ with at least $a(1 -\delta)$ neighbors in $V(H)$. Then for every subset $T$ of $S'$ with $|T|=r-k$ the underlying graph of the blade $\mc{B}[T]$ is $(a,r-k,\delta)$-semicomplete, and so  $\mc{B}[T]$  contains a $K_r$ minor by the assumption of the lemma.	

Let $x= \lfloor|S'|/(r-k) \rfloor$. Then there exists disjoint $T_1,T_2,\ldots T_x \subseteq S'$ such that $|T_i|=r-k$ for $1 \leq i \leq x$.  Let $S''= S  \setminus \cup_{i=1}^x T_i$. Suppose that $|S''| \geq (r-1)(l-x)$. Then there exist disjoint $T_{x+1},\ldots,T_{l} \subseteq S''$, such that $|T_i|=r-1$ for $x+1 \leq i \leq l$. Contracting $H$ to a single vertex gives a model of $K_r$ in $T_i$ for $x+1 \leq i \leq l$. Thus by Lemma~\ref{l:subblade} the blade $\mc{B}$ has an $lK_r$ minor. 

Therefore we may assume for a contradiction that $|S''| \leq (r-1)(l-x)$. We have $|S' \cap S''| \leq r-k-1$ and so the total number of edges of $G$ with one end in $S''$ and another in $V(H)$ is at most
$$a(r-k)+ ((r-1)(l-x)-r-k)a(1-\delta)$$
Adding the edges with one end in $S \setminus S''$, we obtain the following upper bound on the number of edges from $S$ to $V(H)$
\begin{align*}
x(r-k)a+a(r-k)+ ((r-1)(l-x)-r-k)a(1-\delta) \\ 
=a (x(r-k -(r-1)(1-\delta))+ l(r-1)(1-\delta)+\delta(r-k)) \\
= a \left(l(r-k)+\frac{(k-1)(r-k)}{r-1} \right) \\
< a (l(r-k)+k-1), 
\end{align*}
contradicting the assumption that $\mc{B}$ is $(a,a(l(r-k)+k-1))$-hefty. 
\end{proof}

We now have all the ingredients in place for the proofs of our main theorems.

\begin{proof}[Proof of Theorem~\ref{thm:main}  b)]
Let $D$ be as in Lemma~\ref{lem:heftyblade}, let $\eps$ be as in Lemma~\ref{l:semicomplete1}, let $\lambda^*$ be such that every graph $H$ with $d(H) \geq \lambda^* r \sqrt {\log r}$ contains a $K_r$ minor. Assuming $C \gg \lambda^*,D,1/\eps$, we will show that $\c(lK_r) \leq l(r-1)-1$ for all $l \geq C \log r/ \log\log r$.  
	
By Corollary~\ref{c:fanden} it suffices to show that if $\mc{B}'=(G',S')$ is a regula with $d(\mc{B}') > l(r-1)-1$ then $lK_r$ is a minor of $\mc{B}'$. If $|S'| \geq l(r-1)$,  then $\mc{B}'$ has an $lK_r$ minor by Lemma~\ref{l:bladetau}, and so we assume $|S| < l(r-1)$. Therefore $|V(G')-S| \geq 2$, and by Lemma~\ref{lem:heftyblade}, $\mc{B}'$ contains an $(a, (l(r-1)-1- Da)a)$-hefty minor $\mc{B}=(G,S)$ for some integer $a \geq 2$. We will show that $\mc{B}=(G,S)$ contains an  $lK_r$ minor. 

If $a \geq 2\lambda^* r\sqrt{\log r}$ then  $G - S$ has a $K_r$ minor, and so $\mc{B}'$ contains an $nK_r$ minor for any integer $n>0$. Thus we assume \begin{equation}\label{e:asmall}
\eps a \leq 2\lambda^*r\sqrt{\log r}
\end{equation}
Suppose next that  $l \geq 2Da^3$. Then $G$ contains at least  $(l(r -2) + Da^2)a$ edges joining vertices of $S$ to vertices in $V(G)-S$, and so $|S| \geq l(r-2)+Da^2$.  
Moreover, $|S| \leq l(r-1)$, and therefore  at most $Da^2$ vertices in $S$ have a non-neighbor in $V(G)-S$. Thus there exist a set $S' \subseteq S$ such that $|S'| \geq l(r-2)$ and every $v \in S'$ is adjacent to every vertex of $V(G)-S$. Let $S_1,S_2,\ldots,S_l$ be disjoint subsets of $S'$ such that $|S_i|=r-2$ for $1 \leq i \leq l$. Then $\mc{B}[S_i]$ contains $K_r$ as a subgraph, and so $\mc{B}$ has an $lK_r$ minor by Lemma~\ref{l:subblade}. Thus we may assume that $2Da^3 \geq l \geq C$, implying $a\gg 1$, which in turn implies $r\gg 1$ by (\ref{e:asmall}).

Suppose that there exist an integer $2 \leq k\leq r$ such that 
	\begin{align}
	k \cdot\brm \max \left\{\sqrt{\log k},2\frac{\log r}{\log{r}-\log{k}} \right\} < \eps a \label{e:condition11}\\
l(r-1)- 1-Da \geq l(r-k)+k-1, \label{e:condition13}
\end{align}
Then by Lemma~\ref{l:semicomplete1} every $(a,r-k,(k-1)/(r-1))$-semicomplete graph has a $K_r$ minor, and thus by Lemma~\ref{lem:heftyblade1} every  $(a,(l(r-k)+k-1)a)$-hefty blade has an $lK_r$ minor. Meanwhile, the last condition implies that
$\mc{B}$ is $(a,(l(r-k)+k-1)a)$-hefty. Thus it remains to find  $k$ satisfying the above.

Let $ k= \lceil 2Da/l+1 \rceil.$ Then $(k-1)(l-1) \geq Da+1$ and so (\ref{e:condition13}) holds. If $k \leq 3$ then (\ref{e:condition11}) also holds $\eps a, r \gg 1$. Otherwise, $k \leq 4Da/l$. By (\ref{e:asmall}), we have  \begin{align*}\log l &\geq \log C + \log \log r - \log\log\log r \\ &\geq \frac{1}{3} \log \log r -\log r +\log a +    \log 4D,\end{align*} and so $\log k \leq \log r - \frac{1}{3}\log\log r$. Thus the left side of (\ref{e:condition11}) is at most
$$a 
\cdot \frac{4D \log\log r}{C\log r} \cdot \frac{2\log r}{\frac{1}{3}\log\log r} = \frac{24D}{C}a < \eps a$$
as desired.
\end{proof}

\begin{proof}[Proof of Theorem~\ref{thm:upper}] 
	The argument is very similar to the proof of Theorem~\ref{thm:main} b) above, except that we use Lemma~\ref{l:semicomplete2} in place of Lemma~\ref{l:semicomplete1}.
	
	Let $D$ be as in Lemma~\ref{lem:heftyblade}, let $\eps$ be as in Lemma~\ref{l:semicomplete2}, and  let  $\lambda^*$ be such that every graph $H$ with $d(H) \geq \lambda^* r \sqrt{ \log r}$ contains a $K_r$ minor, and let $C$ be as in Theorem~\ref{thm:main} b). We show that the theorem holds as long as $C_u \gg C, \lambda^*,D,1/\eps$. 
	
	Let $\Delta = C_u r \log r /l$. 
As in the proof of Theorem~\ref{thm:main} by Lemma~\ref{lem:heftyblade} it suffices to show that  that if $\mc{B}=(G,S)$ is an $(a, (l(r-1) - 1 + \Delta - Da)a)$-hefty blade for some integer $a \geq 2$ then $\mc{B}$ contains an $lK_r$ minor. By Theorem~\ref{thm:main} b) we may assume that $l \leq C\log r \log \log r$.

As in the previous proof we may assume that $|S| < l(r-1)$ and that (\ref{e:asmall}) holds. The first of these conditions implies $Da \geq \Delta$, that is  \begin{equation}\label{e:alarge}
 a \geq  \frac{C_u r \log r}{Dl}.
 \end{equation} Substituting the upper bound on $l$, we have $a > r/\eps$. As a consequence of (\ref{e:asmall}) amd (\ref{e:alarge}) we have $r \gg 1$ and  
 \begin{equation}\label{e:llarge}
  l >   6D\lambda^*\sqrt{\log r}.
 \end{equation} (The constants in the above inequalities may seem arbitrary, but are chosen for later use.) 

As in the proof of Theorem~\ref{thm:main} successively applying Lemma~\ref{l:semicomplete2} and Lemma~\ref{lem:heftyblade1} we see that it suffices to find a positive integer $k\geq 2$ satisfying
	\begin{align} \brm{max} \{r,\sqrt{rk\log r}\} < \eps a, \label{e:condition21}\\
	0.2 \leq \frac{k-1}{r-1}, \label{e:condition22}\\
   l(r-1) - Da \geq l(r-k)+k-1. \label{e:condition23}
\end{align}

Choose $k = cDa/l$ for some $2 < c <3$. Then $lk \geq 2Da$ and (\ref{e:condition23}) holds. The condition (\ref{e:condition22}) holds by (\ref{e:asmall}) and (\ref{e:llarge}). It remains to show that $\sqrt{rk\log r} < \eps a$, i.e. $$\frac{cDa}{l}r\log r < \eps^2 a^2,$$
which follows directly from (\ref{e:alarge}).
\end{proof}

\section{Concluding remarks}\label{s:conclude}

In this paper we explored applications of the structural lemma of Eppstein~\cite{Eppstein10} to bounds on the asymptotic extremal function $\c(H)$ for disconnected graphs $H$. In particular, the large portion of the paper is dedicated to proving bounds on $\c(lK_r)$. In this direction the following interesting questions remain open

\begin{que}\label{q:1} How large is $\c(2K_r) - \c(K_r)$?
\end{que}	

Clearly, $\c(2K_r) - \c(K_r) \geq 1$, and  we have $\c(2K_r) - \c(K_r) \leq  r-1$ by Theorem~\ref{thm:infty+}, but we can not improve on either of the bounds. Giving a precise answer to Question~\ref{q:1} might be out of reach of the current techniques, as it seems likely to involve obtaining estimates on $c(K_r)$ with additive error sublinear in $r$. In contrast,  we believe that it is possible that a refinement of the tools presented in this paper is sufficient to answer the following two questions.

\begin{que}\label{q:2} Give an estimate on $\c(lK_r)$ which is asymptotically tight for all $l,r$ such that $l+r \to \infty$.
\end{que}	

As noted in the introduction, we have $$\frac{1}{2}-o(1) \leq \frac{\c(lK_r)}{\lambda r \sqrt{\log r} + l(r-1) } \leq 1 + o(1),$$
but can one improve on the estimate in denominator to remove the gap between the bounds?

\begin{que}\label{q:3}  Give a  tight estimate of $\c(lK_r) - l(r-1)$ in the range $l=\omega(\sqrt{\log r})$ and $l=o(\log r/ \log\log r)$. 
\end{que}	

Theorems~\ref{thm:lower} and~\ref{thm:upper} provide bounds on the above difference which differ by a constant factor. We believe that the lower bound is tight.

There are also many natural questions which could be asked about the behaviour of  $\c(lH)$ for non-complete graph $H$. For example, define the \emph{excess} of $H$ by $$\brm{exc}(H)=\lim_{l \to \infty} (\c(lH)-l\tau(H)+1).$$
By (\ref{e:tau}) and Theorem~\ref{thm:infty+}, $\brm{exc}(H)$ is well-defined and is non-negative for every graph $H$. By Theorem~\ref{t:Thomason} we have $\brm{exc}(K_r)=0$ for every $r$. By Theorem~\ref{thm:cycles} we have $\brm{exc}(C_l)=0$ for every $l \neq 4$, while $\brm{exc}(C_4)=1/2$.

\begin{que}\label{q:4}  Describe $\brm{exc}(H)$ in terms of other (natural) parameters of the graph $H$. 
\end{que}	

Finally, note once again that $\c((l+1)H)-\c(lH) \leq \tau(H)$ for all $H$ and all $l \geq 1$ by Theorem~\ref{thm:infty+}. It is possible to show that for fixed $H$ and large enough $l$ the above inequality holds with equality. Hence one might consider the following question.
  
  \begin{que}\label{q:5}  For a fixed graph $H$ is the sequence  $\c((l+1)H)-\c(lH)$ unimodular? Is it non-decreasing?
  \end{que}	
  
  Note that the answer to Question~\ref{q:5} might shed light on Questions~\ref{q:1} and~\ref{q:2}.

\bibliographystyle{alpha}
\bibliography{snorin}
\end{document}